\documentclass[11pt]{article}
\usepackage{graphicx}
\usepackage{amsmath,amsthm,amssymb}
\usepackage{xypic}
\usepackage{mathabx}
\usepackage{bbm}
\usepackage{color}
\def\smallint{\begingroup\textstyle \int\endgroup}
\newtheorem*{theorem*}{Theorem}
\newtheorem*{definition*}{Definition}
\newtheorem*{proposition*}{Proposition}
\newtheorem{theorem}{Theorem}[section]
\newtheorem{proposition}[theorem]{Proposition}
\newtheorem{definition}[theorem]{Definition}
\newtheorem{example}[theorem]{Example}

\newtheorem{remark}[theorem]{Remark}
\newtheorem{corollary}[theorem]{Corollary}
\newtheorem{lemma}[theorem]{Lemma}
\newtheorem{question}[theorem]{Question}

\newcommand\blfootnote[1]{%
  \begingroup
  \renewcommand\thefootnote{}\footnote{#1}%
  \addtocounter{footnote}{-1}%
  \endgroup
}

\begin{document}
\title{THH and traces of enriched categories}
\author{John D. Berman}
\maketitle

\begin{abstract}\blfootnote{The author was supported by an NSF Postdoctoral Fellowship under grant 1803089.}
We prove that topological Hochschild homology (THH) arises from a presheaf of circles on a certain combinatorial category, which gives a universal construction of THH for any enriched $\infty$-category.

Our results rely crucially on an elementary, model-independent framework for enriched higher category theory, which may be of independent interest.
\end{abstract}

\section{Introduction}
\noindent Those interested only in enriched category theory, read Sections 1.3 and 2.

\subsection{Background}
\noindent The topological Hochschild homology (THH) of a ring or ring spectrum $R$ is an object much studied in recent years, because it can be used in favorable cases to compute the algebraic K-theory of $R$. The strategy takes advantage of a rich \emph{cyclotomic} structure on $\text{THH}(R)$, which is a refinement of a natural circle action. This so-called \emph{trace methods} approach to K-theory originated with the Dennis trace in the 70s and Bokstedt's work \cite{Bokstedt} in the 80s, and has taken off since then. A modern account is \cite{Arbeit}.

THH has a variety of other structure in addition to the circle action. First, it is Morita invariant, and therefore lifts to an invariant of $\infty$-categories enriched in spectra. (Even more, THH is an invariant of noncommutative motives \cite{BGT} 10.2.)

If $\mathcal{C}$ is a spectral $\infty$-category, then $\text{THH}(\mathcal{C})$ is the geometric realization of the simplicial spectrum $$\text{THH}(\mathcal{C})_n=\coprod_{X_0,\ldots,X_n\in\mathcal{C}}\text{Hom}(X_0,X_1)\otimes\cdots\otimes\text{Hom}(X_n,X_0),$$ which is called the cyclic bar construction. If $\mathcal{C}$ has a single object, we may identify it with an $\mathbb{E}_1$-ring spectrum $R$, and we recover the original notion of THH as the geometric realization of $$\text{THH}(R)_n=R^{\otimes n+1}.$$ Second, THH is a trace functor \cite{Kaledin} \cite{CampbellPonto}. That is, we may further generalize, defining $\text{THH}(F)$ for any endomorphism $F:\mathcal{C}\rightarrow\mathcal{C}$ of a spectral $\infty$-category. Usual THH is recovered via $\text{THH}(\mathcal{C})=\text{THH}(\text{id}_\mathcal{C})$, and THH satisfies the trace identity $\text{THH}(FG)\cong\text{THH}(GF)$.

Working at this level of generality is useful because the trace identity on THH implies many other useful properties, including Morita invariance \cite{HSS}. It is therefore natural to ask:

\begin{question}\label{Q1}
To what extent does the trace identity uniquely determine THH? That is, to what extent is THH the universal trace functor on the $\infty$-category $\text{Cat}^\text{Sp}$ of spectral $\infty$-categories?
\end{question}

\noindent The trace identity asserts that the functor $\text{THH}:\coprod_\mathcal{C}\text{End}(\mathcal{C})\rightarrow\text{Sp}$ coequalizes (up to homotopy) the diagram $$\coprod_{\mathcal{C}_0,\mathcal{C}_1}\text{Fun}^\text{Sp}(\mathcal{C}_0,\mathcal{C}_1)\otimes\text{Fun}^\text{Sp}(\mathcal{C}_1,\mathcal{C}_0)\rightrightarrows\coprod_{\mathcal{C}}\text{Fun}^\text{Sp}(\mathcal{C},\mathcal{C}).$$ This diagram itself is the 1-skeleton of a cyclic bar construction, taking place in $\text{Cat}^\text{Sp}$ (rather than Sp). We might hope that THH is compatible with the entire cyclic bar construction, not just its 1-skeleton. To formalize this situation, we are forced to study THH of more general enriched $\infty$-categories.

\begin{definition}\label{Def1}
If $\mathcal{V}$ is a symmetric monoidal $\infty$-category\footnote{Since $\text{THH}(\mathcal{C})$ is a colimit, we should either assume $\mathcal{V}$ is presentable, or define $\text{THH}(\mathcal{C})$ as a presheaf on $\mathcal{V}$. We will ignore this unimportant subtlety in the introduction.} and $\mathcal{C}$ is a $\mathcal{V}$-enriched category, $\text{THH}(\mathcal{C})\in\mathcal{V}$ is the geometric realization of the cyclic bar construction (in $\mathcal{V}$): $$\text{THH}(\mathcal{C})_n=\coprod_{X_0,\ldots,X_n\in\mathcal{C}}\text{Hom}(X_0,X_1)\otimes\cdots\otimes\text{Hom}(X_n,X_0).$$
\end{definition}

\begin{definition}
If $\mathcal{C}$ is a $\mathcal{V}$-enriched category and $X\in\mathcal{V}$, a \emph{homotopy-coherent trace functor} from $\mathcal{C}$ to $X$ is a morphism $\text{THH}(\mathcal{C})\rightarrow X$.
\end{definition}

\noindent Then we might ask, en route to answering Question \ref{Q1}:

\begin{question}
Can THH be promoted to a homotopy-coherent trace functor $\text{THH}(\text{Cat}^\text{Sp})\xrightarrow{\text{THH}}\text{Sp}$?
\end{question}

\noindent We will not answer these questions. However, if we hope to study questions of a formal nature like these, Definition \ref{Def1} is not ideal. It is fundamentally a calculation of THH, when we would prefer a universal property. For example, it obscures the circle action (and ensuing cyclotomic structure) on THH, and it makes explicit reference to the set of objects of $\mathcal{C}$.

It also relies implicitly on enriched $\infty$-categories, for which the state of the art (largely due to Gepner and Haugseng \cite{GH}) is rather technical and dependent on the particular model of quasicategories.

Motivated by these objections, we will:
\begin{enumerate}
\item present a combinatorial, model-independent framework for enriched higher category theory; a $\mathcal{V}$-enriched category is a symmetric monoidal functor $\text{Bypass}_S\rightarrow\mathcal{V}$ from a category of combinatorial graphs;
\item present a universal construction of THH, using this framework; THH is the pushforward along $\text{Bypass}_S\rightarrow\mathcal{V}$ of a certain presheaf $\mathcal{O}_\text{thh}$ on $\text{Bypass}_S$ obtained as a push-pull construction applied to the circle;
\item explicitly calculate the presheaf $\mathcal{O}_\text{thh}$.
\end{enumerate}

\subsection{First results}
\noindent As a warm-up, a special case of (2) can be stated without any enriched category theory.

Suppose that $A$ is an associative algebra in $\mathcal{V}$ (equivalently, a $\mathcal{V}$-enriched category with one object). We regard $A$ as a symmetric monoidal functor $$\text{Ass}\xrightarrow{A}\mathcal{V},$$ where Ass is the \emph{associative PROP} (\cite{HA} 4.1.1): An object of Ass is a finite set, and a morphism a function $f:X\rightarrow Y$ with a total ordering of each $f^{-1}(y)$.

Let $\Lambda$ be Connes' cyclic category \cite{Connes}; roughly, the category of cyclically ordered finite sets. There is a functor $\Lambda\xrightarrow{k}\text{Ass}$ which forgets the cyclic ordering (see Proposition \ref{PropOthh}). Moreover, the classifying space of $\Lambda$ is $BS^1$ \cite{Connes}, so there is a functor of $\infty$-categories $\Lambda\xrightarrow{r}BS^1$ which formally inverts all the morphisms of $\Lambda$. In summary, we have the diagram: $$\xymatrix{
\Lambda\ar[r]^k\ar[d]_r &\text{Ass}\ar[r]^A &\mathcal{V} \\
BS^1. &&
}$$

\begin{theorem*}[\ref{ThmCase1}]
If $\mathcal{V}$ is presentable and $A$ is an associative algebra in $\mathcal{V}$, $$\text{THH}(A)\cong A_\ast k_\ast r^\ast(S^1).$$
\end{theorem*}

\noindent The notation requires some explanation:
\begin{itemize}
\item $S^1$ is the circle with the free $S^1$-action, regarded as an $S^1$-space and thus a presheaf $(BS^1)^\text{op}\rightarrow\text{Top}$;
\item $r^\ast:\mathcal{P}(BS^1)\rightarrow\mathcal{P}(\Lambda)$ denotes precomposition of a presheaf by $r$;
\item $k_\ast:\mathcal{P}(\Lambda)\rightarrow\mathcal{P}(\text{Ass})$ denotes left Kan extension along $k$;
\item $A_\ast:\mathcal{P}(\text{Ass})\rightarrow\mathcal{V}$ is the unique functor extending $A$ to $\mathcal{P}(\text{Ass})\supseteq\text{Ass}$ which preserves small colimits.
\end{itemize}

\noindent The upshot is that we have a universal construction of THH which makes explicit the $S^1$-action. In particular, there is a presheaf $\mathcal{O}_\text{thh}=k_\ast r^\ast S^1$ on the associative PROP Ass, described by a push-pull procedure applied to the circle, and $$\text{THH}(\mathcal{C})\cong\mathcal{C}_\ast\mathcal{O}_\text{thh}.$$ Our main results will be:

\begin{itemize}
\item (Theorem \ref{ThmCase2}) a generalization of Theorem \ref{ThmCase1}, replacing the associative algebra by any $\mathcal{V}$-enriched category;
\item (Theorem \ref{MainThm}) an explicit calculation of the presheaf $\mathcal{O}_\text{thh}$.
\end{itemize}

\begin{remark}
We will extend Theorem \ref{ThmCase1} by categorification, replacing associative algebras by enriched categories. We conjecture that it may also be generalized in other directions. For instance, we believe there is an analogue calculating factorization homology $\int_M A$ when $M$ is an $n$-manifold and $A$ is an $\mathbb{E}_n$-algebra.

If so, Theorem \ref{ThmCase1} would recover the well-known identification between THH and factorization homology over $S^1$. Ayala-Mazel-Gee-Rozenblyum \cite{AMGR} have results related to ours in the factorization homology setting.
\end{remark}

\subsection{Enriched categories}
\noindent We will use the following straightforward definition of enriched $\infty$-categories. For each set $S$, there is a symmetric monoidal category $\text{Bypass}_S$, and:

\begin{definition*}[\ref{DefEnr}]
If $\mathcal{V}$ is a symmetric monoidal $\infty$-category, a $\mathcal{V}$-enriched category with set $S$ of objects is a symmetric monoidal functor $$\mathcal{C}:\text{Bypass}_S\rightarrow\mathcal{V}.$$
\end{definition*}

\noindent We will now describe $\text{Bypass}_S$. An object is a directed graph on the fixed set $S$ of vertices. These are really \emph{multigraphs}, in that they may include multiple edges between the same two vertices, as well as loops from a vertex to itself.

A morphism in $\text{Bypass}_S$ is a combination of the following combinatorial moves that we call \emph{bypass operations}:
\begin{itemize}
\item choose an ordered set of edges forming a path $X_0\to X_1\to\cdots\to X_n$, and replace them by a single edge $X_0\to X_n$;
\item choose a vertex $X$, and add a new edge (loop) $X\to X$.
\end{itemize}

\noindent If $\Gamma,\Gamma^\prime\in\text{Bypass}_S$, we write $\Gamma\otimes\Gamma^\prime$ for the graph whose set of edges is the disjoint union of edges in $\Gamma$ and edges in $\Gamma^\prime$. In this way, $\text{Bypass}_S$ is symmetric monoidal, and the unit is the empty graph $\emptyset$ (no edges).

We also write $(X,Y)\in\text{Bypass}_S$ for the graph with a single edge from $X$ to $Y$. Essentially by construction, $\text{Bypass}_S$ admits a presentation as a symmetric monoidal category by:
\begin{itemize}
\item objects $(X,Y)$ for $X,Y\in S$;
\item morphisms $(X,Y)\otimes(Y,Z)\rightarrow(X,Z)$ for $X,Y,Z\in S$;
\item morphisms $\emptyset\rightarrow(X,X)$ for $X\in S$;
\item associative and unital relations.
\end{itemize}

\noindent This presentation encodes the classical axioms of an enriched category. The notation has been set up conveniently so that, if $\mathcal{C}:\text{Bypass}_S\rightarrow\mathcal{V}$ is an enriched category, $\mathcal{C}(X,Y)$ is the object of morphisms from $X$ to $Y$.

In Section 2, we will prove:

\begin{proposition*}[\ref{PropGH}]
Definition \ref{DefEnr} agrees with that of Gepner-Haugseng \cite{GH}.
\end{proposition*}

\noindent Now suppose we have a graph $\Gamma\in\text{Bypass}_S$. An \emph{Eulerian tour} on $\Gamma$ is a cyclic ordering on the edges of $\Gamma$ so that they form a single cycle. We let $\text{Bypass}_S^\text{Eul}$ denote the category of nonempty graphs in $\text{Bypass}_S$ with specified Eulerian tour, and $\text{Bypass}_S^\text{Eul}\xrightarrow{k}\text{Bypass}_S$ the forgetful functor.

Our first main result generalizes Theorem \ref{ThmCase1} to a universal construction of THH of enriched categories:

\begin{theorem*}[\ref{ThmCase2}]
If $\mathcal{C}:\text{Bypass}_S\rightarrow\mathcal{V}$ is a $\mathcal{V}$-enriched category with set $S$ of objects, $$\text{THH}(\mathcal{C})\cong\mathcal{C}_\ast k_\ast r^\ast(S^1),$$ with maps as in the diagram $$\xymatrix{
\text{Bypass}_S^\text{Eul}\ar[d]_r\ar[r]^k &\text{Bypass}_S\ar[r]^{\mathcal{C}} &\mathcal{V} \\
BS^1. &&
}$$
\end{theorem*}

\noindent The functor $r$ is that which exhibits $BS^1$ as the classifying space of $\text{Bypass}_S^\text{Eul}$ (Corollary \ref{CorCS}).

In the case $S=\ast$, then $\text{Bypass}_\ast=\text{Ass}$, because a graph on one vertex can be identified with a finite set (of loops at that vertex). On the other hand, an Eulerian tour is a cyclic ordering, so $\text{Bypass}_\ast^\text{Eul}=\Lambda$, and we recover Theorem \ref{ThmCase1}.

\begin{remark}
The description $\text{THH}(\mathcal{C})\cong\mathcal{C}_\ast k_\ast r^\ast(S^1)$ has a few benefits: Most obviously, it makes explicit the $S^1$-action. It also isolates the cyclic bar construction as $k_\ast r^\ast(S^1)$, divorcing it from the enriched category theory (which is encoded in $\mathcal{C}_\ast$). Formal arguments involving the cyclic bar construction can now be encoded as properties of the presheaf $k_\ast r^\ast(S^1)$.
\end{remark}

\subsection{Calculations}
\noindent If $\mathcal{C}:\text{Bypass}_S\rightarrow\mathcal{V}$ is an enriched category with set $S$ of objects, $\text{THH}(\mathcal{C})$ is a colimit of terms which can be defined in $\text{Bypass}_S$. Therefore, there is a formal colimit in $\text{Bypass}_S$, namely the geometric realization of $$(\mathcal{O}_\text{thh})_n=\coprod_{X_0,\ldots,X_n}(X_0,X_1)\otimes\cdots\otimes(X_n,X_0),$$ for which $\mathcal{C}(\mathcal{O}_\text{thh})\cong\text{THH}(\mathcal{C})$. By an extension of the Yoneda lemma, formal colimits can be identified with presheaves of spaces $\text{Bypass}_S^\text{op}\rightarrow\text{Top}$, so we think of $\mathcal{O}_\text{thh}$ as a presheaf.

Then Theorem \ref{ThmCase2} can be restated $\mathcal{O}_\text{thh}\cong k_\ast r^\ast(S^1)$. This is a useful universal property of THH. However, we can also do more: We can explicitly calculate the presheaf $\mathcal{O}_\text{thh}:\text{Bypass}_S^\text{op}\rightarrow\text{Top}$:

\begin{proposition*}[Corollary \ref{CorOthh}]
We have $$\mathcal{O}_\text{thh}(\Gamma)\cong\left\{\begin{array}{lr}
(S^1)^{\amalg\text{Eul}(\Gamma)}, &\text{if }\Gamma\neq\emptyset \\
S, &\text{if }\Gamma=\emptyset
\end{array}\right.,$$ where the disjoint union of circles is taken over the set $\text{Eul}(\Gamma)$ of Eulerian tours on $\Gamma$, and $S$ is the ambient set of objects.
\end{proposition*}

\begin{remark}\label{RmkRes}
Because of the different behavior at $\emptyset$, we will clean up our exposition by restricting away from the empty graph.
\end{remark}

\noindent Let $\mathcal{O}_\text{thh}^{+}$ denote $\mathcal{O}_\text{thh}$ restricted away from the empty graph. This is a presheaf on the full subcategory $\text{Bypass}_S^{+}\subseteq\text{Bypass}_S$ of nonempty graphs.

The last proposition describes $\mathcal{O}_\text{thh}^{+}(\Gamma)$ for each $\Gamma$; however, the restriction maps $\mathcal{O}_\text{thh}^{+}(f)$ for $f:\Gamma\rightarrow\Gamma^\prime$ may be nontrivially twisted.

The following theorem will give a complete description of $\mathcal{O}_\text{thh}^{+}$. If $\text{Eul}(\Gamma)$ denotes the set of Eulerian tours on $\Gamma$, then Eul is a presheaf of sets defined on nonempty graphs $\text{Bypass}_S^{+}$.

\begin{theorem*}[\ref{MainThm}]
If $\mathcal{O}_\text{thh}^{+}$ is as above, then:
\begin{enumerate}
\item $\mathcal{O}_\text{thh}^{+}$ has a canonical $S^1$-action;
\item $(\mathcal{O}_\text{thh}^{+})_{hS^1}\cong\text{Eul}$;
\item The moduli space of presheaves satisfying (1)-(2) is $\mathbb{Z}\times BS^1$; that is, such presheaves are determined up to equivalence by an integer invariant we call \emph{degree};
\item The degree of $\mathcal{O}^{+}_\text{thh}$ is $1$.
\end{enumerate}
\end{theorem*}

\noindent Notice that the last proposition (Corollary \ref{CorOthh}) follows from the theorem. Indeed, any $\mathcal{O}$ satisfying (1)-(2) evaluates on $\Gamma\neq\emptyset$ by $$\mathcal{O}(\Gamma)\cong(S^1)^{\amalg\text{Eul}(\Gamma)}.$$ Property (4) describes the twisting of the sheaf $\mathcal{O}_\text{thh}$.

In general, write $\text{Eul}(n)$ for the $S^1$-bundle over Eul of degree $n$.

\begin{corollary}
If $\mathcal{C}$ is $\mathcal{V}$-enriched, then $$\mathcal{C}_\ast(\text{Eul}(\pm 1))\cong\text{THH}(\mathcal{C}).$$ We can also identify $\mathcal{C}_\ast(\text{Eul}(n))$ for all $n$: $$\mathcal{C}_\ast(\text{Eul}(\pm n))\cong\text{THH}(\mathcal{C})_{hC_n},$$ $$\mathcal{C}_\ast(\text{Eul}(0))\cong S^1\otimes\text{THH}(\mathcal{C})_{hS^1},$$ $$\mathcal{C}_\ast(\text{Eul})\cong\text{THH}(\mathcal{C})_{hS^1}.$$
\end{corollary}

\noindent Theorem \ref{MainThm} is the most technical part of the paper. It relies on a key combinatorial lemma that we sketch now.

Namely, the following data is equivalent:
\begin{itemize}
\item a graph $\Gamma\in\text{Bypass}_S$ with a chosen Eulerian tour;
\item a cyclically ordered set of edges $E\in\Lambda$ with a labeling of its vertices in $S$.
\end{itemize}

\noindent Categorically, this combinatorial argument implies that $\text{Bypass}_S^\text{Eul}$ is both the (right fibrational) Grothendieck construction applied to the presheaf Eul on $\text{Bypass}_S$, as well as the (left fibrational) Grothendieck construction applied to a certain presheaf on $\Lambda$.

In the case $S=\ast$, this recovers the observation $\text{Bypass}_\ast^\text{Eul}\cong\Lambda$.

We will use this combinatorial argument to prove that $\text{Bypass}_S^\text{Eul}$ has classifying space $BS^1$ (Corollary \ref{CorCS}), and to pass back and forth between presheaves on $\text{Bypass}_S$ and presheaves on $\text{Bypass}_S^\text{Eul}$ (Corollary \ref{CorComb}). This is the technical material necessary to prove Theorem \ref{MainThm}(3), and the rest follows by Theorem \ref{ThmCase2}.

\subsection{Organization}
\noindent Section 2 concerns enriched higher categories. We want to emphasize that our construction is relatively elementary and model-independent. For that reason, this section is written for a wide audience, including those who may not be interested in THH.

In Section 3, we review the cyclic category $\Lambda$. The results are not new, but there is one crucial idea (Lemma \ref{LemLambda}): There are canonical equivalences $\Lambda\cong\text{Bypass}_\ast^\text{Eul}\cong\Lambda^\text{op}$. Throughout the paper, we will frequently make the second identification $\text{Bypass}_\ast^\text{Eul}\cong\Lambda^\text{op}$.

In the very short Section 4, we define THH of an enriched category via the cyclic bar construction. The cyclic bar construction itself is formally encoded by a presheaf $\mathcal{O}_\text{thh}$ on $\text{Bypass}_S$.

Section 5 contains the proof of the first main result (Theorem \ref{ThmCase1}), which provides a universal construction for THH of associative algebras.

Although the cyclic bar construction of $\text{THH}(\mathcal{C})$ makes reference to the set of objects of $\mathcal{C}$, THH does not actually depend in any meaningful way on the object-set. In Section 6, we prove a weak version of this statement, and use it to generalize our construction of THH to the case of enriched $\infty$-categories.

The final two sections are the calculation, Theorem \ref{MainThm}. For this, we need the combinatorial counting-in-two-ways argument described above, which is in Section 7, followed by the proof of Theorem \ref{MainThm} and its corollaries in Section 8.

\subsection{Acknowledgments}
\noindent This paper benefited from conversations with Rok Gregoric, Aaron Mazel-Gee, Sam Raskin, Chris Schommer-Pries, and Jay Shah. Special thanks to Andrew Blumberg for his regular feedback and interest in this project.

\subsection{Not included in this paper}
\noindent Everything in this paper is done for enriched categories with fixed sets of objects. That is, we define THH as a functor on $\text{Cat}^\mathcal{V}_S$, the $\infty$-category of $\mathcal{V}$-enriched categories with set $S$ of objects (and functors between them which act as the identity on objects).

In fact, THH is functorial on $\text{Cat}^\mathcal{V}$ (the $\infty$-category of small $\mathcal{V}$-enriched categories). The author has chosen not to include a discussion of the functoriality of THH because it would lengthen the paper and distract from the main results.

We hope to include these details in a sequel on cyclotomic structures. Until then, the interested reader can derive the functoriality of THH from the two properties:
\begin{itemize}
\item If $f:S\rightarrow T$ is a function inducing $F:\text{Bypass}_S\rightarrow\text{Bypass}_T$, then $F_\ast\mathcal{O}_\text{thh}\cong\mathcal{O}_\text{thh}$. Hence, if $\mathcal{C}\rightarrow\mathcal{D}$ is fully faithful, there is a functor $\text{THH}(\mathcal{C})\rightarrow\text{THH}(\mathcal{D})$.
\item If $f$ is surjective, then also $F^\ast\mathcal{O}_\text{thh}\cong\mathcal{O}_\text{thh}$. Hence, if $\mathcal{C}\rightarrow\mathcal{D}$ is also essentially surjective, then $\text{THH}(\mathcal{C})\rightarrow\text{THH}(\mathcal{D})$ is an equivalence.
\end{itemize}

\subsection{Notation}\label{SNot}
\noindent We use $\infty$-categorical language throughout, writing Top for the $\infty$-category of spaces, or homotopy types, and $\mathcal{P}(\mathcal{C})=\text{Fun}(\mathcal{C}^\text{op},\text{Top})$ for presheaves on a small $\infty$-category $\mathcal{C}$.

Via the Yoneda embedding $\mathcal{C}\subseteq\mathcal{P}(\mathcal{C})$, we will identify $\mathcal{C}$ with the $\infty$-category of representable presheaves. That is, we will use the same symbol to refer to an object $X\in\mathcal{C}$ or the associated representable presheaf $X\in\mathcal{P}(\mathcal{C})$.

By \cite{HTT} 5.1.5.6, $\mathcal{P}(\mathcal{C})$ can also be identified with the $\infty$-category of \emph{formal colimits} in $\mathcal{C}$. Hence, if $F:\mathcal{C}\rightarrow\mathcal{D}$ is a functor, $\mathcal{C}$ is a small $\infty$-category, and $\mathcal{D}$ is a presentable $\infty$-category, then there is an essentially unique extension $$F_\ast:\mathcal{P}(\mathcal{C})\rightarrow\mathcal{D}$$ which preserves colimits and restricts to $F$ on representables. Also, $F_\ast$ has a right adjoint $F^\ast$ by the adjoint functor theorem \cite{HTT} 5.5.2.9.

Here is a special case: If $F:\mathcal{C}\rightarrow\mathcal{D}$ is a functor between small $\infty$-categories, we embed $\mathcal{D}\subseteq\mathcal{P}(\mathcal{D})$, which is presentable, so we can apply the above construction to produce an adjoint pair $$F_\ast:\mathcal{P}(\mathcal{C})\leftrightarrows\mathcal{P}(\mathcal{D}):F^\ast.$$ As before, the left adjoint $F_\ast$ is the unique functor preserving colimits which restricts to $F$ on representables.

The right adjoint $F^\ast$ is precomposition of a presheaf $\mathcal{D}^\text{op}\rightarrow\text{Top}$ by $F:\mathcal{C}^\text{op}\rightarrow\mathcal{D}^\text{op}$.

\begin{example}\label{ExColim}
For any $\infty$-category $\mathcal{C}$, let $c:\mathcal{C}\rightarrow\ast$ denote the trivial functor. Then $c^\ast:\text{Top}\rightarrow\mathcal{P}(\mathcal{C})$ sends a space $X$ to the constant presheaf at $X$. Hence its left adjoint is $$c_\ast(\mathcal{O})\cong\text{colim}_{\mathcal{C}^\text{op}}(\mathcal{O}).$$
\end{example}

\begin{example}
If $\mathcal{C}$ is an $\infty$-category, it has a classifying space $|\mathcal{C}|$, which has the following universal property: If $X$ is an $\infty$-groupoid, any functor $\mathcal{C}\rightarrow X$ factors uniquely through $|\mathcal{C}|$.

If $\mathcal{C}$ is an ordinary category, $|\mathcal{C}|$ is classically the geometric realization of the nerve of $\mathcal{C}$.

By \cite{HTT} 3.3.4.6, $|\mathcal{C}|\cong c_\ast(\ast)$, where $c:\mathcal{C}\rightarrow\ast$ as before, and $\ast$ is the trivial constant presheaf on $\mathcal{C}$.
\end{example}

\section{Enriched categories}
\noindent In this section, we give a combinatorial description of enriched categories as symmetric monoidal functors $\text{Bypass}_S\rightarrow\mathcal{V}$, where $\text{Bypass}_S$ is a category of graphs and bypass operations. Specifically:

\begin{definition}
A \emph{directed multigraph} (or just \emph{graph}) $\Gamma$ consists of a set $E$ of edges, a set $S$ of vertices, and two functions $s,t:E\rightarrow S$ called the source and target. We say $\Gamma$ is \emph{finite} if $E$ is finite (even if $S$ is not).

Given two directed multigraphs $\Gamma,\Gamma^\prime$ with identical sets of vertices, a \emph{bypass operation} $f:\Gamma\rightarrow\Gamma^\prime$ is a function $f:E\rightarrow E^\prime$ along with a total ordering of the set $f^{-1}(e)$ for each $e\in E^\prime$, satisfying the properties:
\begin{itemize}
\item If $f^{-1}(e)$ is empty, then $e$ is a \emph{loop}; that is, $s(e)=t(e)$;
\item If $f^{-1}(e)=\{e_1<\cdots<e_k\}$, then $e_1\to\cdots\to e_k$ is a \emph{path} from $s(e)$ to $t(e)$; that is, $t(e_i)=s(e_{i+1})$ for $1\leq i<k$, $s(e_1)=s(e)$, and $t(e_k)=t(e)$.
\end{itemize}
\end{definition}

\noindent We think of a graph $\Gamma$ as a set $S$ of vertices and a set $E$ of edges, such that each $e\in E$ is a directed edge from $s(e)$ to $t(e)$. A bypass operation is a series of operations of the following forms (corresponding respectively to the two properties above) which transform $\Gamma$ into $\Gamma^\prime$:
\begin{itemize}
\item Add a loop (an edge from a vertex to itself);
\item Given edges which form a path $e_1\to\cdots\to e_n$, replace all of them by a single edge from $s(e_1)$ to $t(e_n)$.
\end{itemize}

\noindent The second operation is the origin of our term \emph{bypass}.

\begin{definition}
Given a set $S$ (not necessarily finite), $\text{Bypass}_S$ is the category of \emph{finite} directed multigraphs with fixed object set $S$, and bypass operations for morphisms.
\end{definition}

\noindent Given two directed multigraphs $\Gamma,\Gamma^\prime$ on a fixed set $S$ of vertices, denote by $\Gamma\otimes\Gamma^\prime$ the graph with edge set $E\amalg E^\prime$ and the induced source and target maps. Then $\otimes$ is a symmetric monoidal operation on $\text{Bypass}_S$. The unit of the symmetric monoidal structure is the empty graph $\emptyset$ with no edges.

We will introduce notation for a few special graphs that will be important in all that follows:

\begin{definition}
If $X_i\in S$, then $(X_1,X_2)\in\text{Bypass}_S$ denotes the graph with a single edge from $X_1$ to $X_2$. More generally, $$(X_1,X_2,\ldots,X_n)=(X_1,X_2)\otimes(X_2,X_3)\otimes\cdots\otimes(X_{n-1},X_n)$$ denotes the graph with a single path $X_1\to\cdots\to X_n$.
\end{definition}

\noindent As a symmetric monoidal category, $\text{Bypass}_S^\otimes$ admits a presentation by:
\begin{itemize}
\item objects $(X,Y)$,
\item morphisms $(X,Y,Z)=(X,Y)\otimes(Y,Z)\rightarrow(X,Z)$;
\item morphisms $\emptyset\rightarrow(X,X)$;
\item relations (commuting diagrams) $$\xymatrix{
(a,b,c,d)\ar[r]\ar[d] &(a,b,d)\ar[d] &(a,b,b)\ar[r] &(a,b) &(a,a,b)\ar[l] \\
(a,c,d)\ar[r] &(a,d) &(a,b)\ar[u]\ar[ru] &&(a,b)\ar[u]\ar[lu]
}$$
\end{itemize}

\noindent Notice that this presentation corresponds exactly to the axioms of an enriched category. Hence, the following definition agrees with the usual one when $\mathcal{V}$ is an ordinary category:

\begin{definition}\label{DefEnr}
If $\mathcal{V}$ is a symmetric monoidal $\infty$-category, a $\mathcal{V}$-enriched category with set $S$ of objects is a symmetric monoidal functor $$\mathcal{C}:\text{Bypass}_S\rightarrow\mathcal{V}.$$
\end{definition}

\begin{example}
If $S=\ast$ is a singleton, then $\text{Bypass}_{\ast}$ is the category of finite sets, whose morphisms come with prescribed total orderings of the fibers. This is the symmetric monoidal envelope of the associative operad (\cite{HA} 4.1.1.1), so symmetric monoidal functors $\text{Bypass}_{\ast}\rightarrow\mathcal{V}$ may be identified with associative algebras in $\mathcal{V}$.

This is consistent with the principle that associative algebras are enriched categories with one object.
\end{example}

\begin{remark}
The notation has been set up conveniently so that when the functor $\mathcal{C}$ is evaluated at $(X,Y)\in\text{Bypass}_S$, we get $\mathcal{C}(X,Y)\cong\text{Map}(X,Y)$. If 1 denotes the unit in $\mathcal{V}$, we also have $$\mathcal{C}(\emptyset)\cong 1,$$ $$\mathcal{C}(X_1,\ldots,X_n)\cong\mathcal{C}(X_1,X_2)\otimes\cdots\otimes\mathcal{C}(X_{n-1},X_n).$$
\end{remark}

\noindent We end this section with a comparison to the definition of Gepner-Haugseng:

\begin{proposition}\label{PropGH}
Definition \ref{DefEnr} agrees with that of Gepner-Haugseng \cite{GH}.
\end{proposition}

\begin{proof}
Gepner and Haugseng (\cite{GH} 2.2.17) define $\text{Cat}^\mathcal{V}_S=\text{Alg}_{\mathcal{O}_S^\otimes}(\mathcal{V})$, where $\mathcal{O}_S^\otimes$ is a nonsymmetric $\infty$-operad defined in \cite{GH} 2.1. For them, $\mathcal{V}$ is just monoidal, but for us it is even \emph{symmetric} monoidal. Hence $\text{Cat}^\mathcal{V}_S\cong\text{Alg}_{\bar{\mathcal{O}}_S^\otimes}(\mathcal{V})$, where $\bar{\mathcal{O}}_S^\otimes$ is the \emph{symmetrization} of $\mathcal{O}_S^\otimes$ (\cite{GH} Definition 3.7.6).

Unpacking definitions, the colors of $\bar{\mathcal{O}}_S^\otimes$ are symbols $(X,Y)$ for $X,Y\in S$, and active morphisms (multilinear morphisms) from the unordered tuple $((X_1,Y_1),\ldots,(X_n,Y_n))$ to $(A,B)$ are canonically in bijection with $\text{Bypass}_S$-morphisms of the form $$(X_1,Y_1)\otimes\cdots\otimes(X_n,Y_n)\rightarrow(A,B).$$ The symmetric monoidal envelope $\text{Env}(\bar{\mathcal{O}}_S^\otimes)$ is the subcategory of $\bar{\mathcal{O}}_S^\otimes$ spanned by just the active morphisms (\cite{HA} 2.2.4). Because objects of $\text{Bypass}_S$ can be written in a unique way as tensor products of the elementary objects $(X,Y)$, unpacking definitions shows that $\text{Env}(\bar{\mathcal{O}}_S^\otimes)\cong\text{Bypass}_S$, as symmetric monoidal categories. Verification is elementary because both of these are 1-categories.

The theorem is then the universal property of symmetric monoidal envelopes (\cite{HA} 2.2.4.9).
\end{proof}

\noindent The reader may object: We have defined a $\mathcal{V}$-enriched category, but a bit more work is needed to define enriched functors (that is, to define an $\infty$-category $\text{Cat}^\mathcal{V}$). This won't be important for this paper, but the construction is a standard one, which we sketch below. See \cite{GH} for details.

We build $\text{Cat}^\mathcal{V}$ out of $\text{Cat}^\mathcal{V}_S$ as follows. The construction $S\mapsto\text{Bypass}_S$ is functorial $\text{Bypass}_{(-)}:\text{Set}\rightarrow\text{Cat}$, and the associated Grothendieck construction (or cocartesian fibration) is $p:\text{preCat}^\mathcal{V}\rightarrow\text{Set}$.

In $\text{preCat}^\mathcal{V}$, the objects are enriched categories, and morphisms are composites $F:\mathcal{C}\rightarrow\text{im}(F)\subseteq\mathcal{D}$, where the first functor acts as the identity on objects, and the second is fully faithful (that is, a $p$-cocartesian morphism in $\text{preCat}^\mathcal{V}$).

Now we need to insist that fully faithful essentially surjective functors are equivalences. We do this by inverting those $p$-cocartesian morphisms $f$ in $\text{preCat}^\mathcal{V}$ for which $p(f)$ is a surjection. The result is $\text{Cat}^\mathcal{V}$.

\begin{remark}
Actually, Gepner-Haugseng's enriched categories have underlying \emph{spaces} of objects, while ours have underlying \emph{sets} of objects. By \cite{GH} 5.3.17, these two theories are equivalent.
\end{remark}

\noindent See Section 1.7 for a comment on how THH interacts with the construction of $\text{Cat}^\mathcal{V}$.

\section{The cyclic category}
\noindent In this section, we will review the cyclic category with an eye towards our application to enriched category theory.

\begin{definition}
If $\Gamma$ is a \emph{nonempty} directed multigraph, an \emph{Eulerian tour} on $\Gamma$ is a total ordering of the edges in such a way that they form a single cycle. (That is, it is a path which begins and ends at the same vertex, visiting each edge exactly once.)

We will regard two Eulerian tours on $\Gamma$ as the same if they differ only by cyclic permutation.
\end{definition}

\noindent Suppose that $\Gamma\xrightarrow{f}\Gamma^\prime$ is a bypass operation between nonempty graphs. Given an Eulerian tour on $\Gamma^\prime$, then the total orderings of fibers of $f$ (part of the data of the morphism $f$) induces an Eulerian tour on $\Gamma$. A cyclic shift of the tour on $\Gamma^\prime$ induces a cyclic shift of the induced tour on $\Gamma$.

\begin{definition}
Let $\text{Bypass}_S^\text{Eul}$ denote the category of (finite, nonempty) directed multigraphs with a chosen Eulerian tour, along with those morphisms $\Gamma\rightarrow\Gamma^\prime$ in $\text{Bypass}_S$ which pull back the chosen Eulerian tour on $\Gamma^\prime$ to the chosen Eulerian tour on $\Gamma$.
\end{definition}

\noindent To understand the structure of $\text{Bypass}_S^\text{Eul}$, we first consider the case $S=\ast$. In this case, $\text{Bypass}_\ast^\text{Eul}$ may be identified with Connes' cyclic category.

\begin{definition}[Connes' cyclic category \cite{Connes}] 
Let $\mathbb{T}_n$ denote the category on objects $v_i$ generated by \emph{irreducible morphisms} $e_i$: $$v_0\xrightarrow{e_0}\cdots\xrightarrow{e_{n-1}}v_n\xrightarrow{e_n}v_0.$$ We call this category a \emph{cyclically ordered set}. Each object has a \emph{degree 1} endomorphism given by passing around this cycle once.

The cyclic category $\Lambda$ is the category of finite cyclically ordered sets and functors between them which send degree 1 endomorphisms to degree 1 endomorphisms. We call these \emph{degree 1 functors}.
\end{definition}

\noindent There are two functors $$i_v:\Lambda\rightarrow\text{Bypass}_\ast^\text{Eul},$$ $$i_e:\Lambda^\text{op}\rightarrow\text{Bypass}_\ast^\text{Eul},$$ where $i_v(\mathbb{T}_n)$ is the graph on one vertex whose edges are given by the objects $v_i$ of $\mathbb{T}_n$, and $i_e(\mathbb{T}_n)$ is the graph whose edges are given by the irreducible morphisms $e_i$ of $\mathbb{T}_n$.

These two functors treat morphisms as follows: If $f:\mathbb{T}_m\rightarrow\mathbb{T}_n$, then $i_v(f)$ acts as the function $f$. On the other hand, $i_e(f)$ sends an edge $e\in\mathbb{T}_n$ to the unique edge $e^\prime\in\mathbb{T}_m$ for which $f(e^\prime)$, written as a composite of elementary morphisms, includes the morphism $e$.

\begin{lemma}\label{LemLambda}
Both $i_v$ and $i_e$ are equivalences of categories.
\end{lemma}

\noindent In particular, we have a canonical equivalence $\Lambda\cong\Lambda^\text{op}$, known from \cite{Connes}.

\begin{proof}
Let $\Gamma\in\text{Bypass}_S^\text{Eul}$ (notice $S$ is not necessarily $\ast$). The Eulerian tour endows the edges of $\Gamma$ with a cyclic ordering, so there is a forgetful functor $\text{Bypass}_S^\text{Eul}\rightarrow\Lambda$ which remembers just the set of edges. When $S=\ast$, this is $i_v^{-1}$, so $i_v$ is an equivalence.

On the other hand, suppose that $\Gamma,\Gamma^\prime\in\text{Bypass}_S^\text{Eul}$ have cyclically ordered sets of edges $\mathbb{T}_m,\mathbb{T}_n$, respectively. Given a map $f:\Gamma\rightarrow\Gamma^\prime$, there is an induced map $f^\ast:\mathbb{T}_n\rightarrow\mathbb{T}_m$ which sends an edge $e\in\Gamma^\prime$ to the unique edge $e^\prime\in\Gamma$ such that $$f(e^\prime)\geq e>f(e^\prime_{-}).$$ Here, $e^\prime_{-}$ is the edge immediately preceding $e^\prime$ in the cyclic ordering.

The significance is that $e^\prime$ and $e$ have the same source vertex, so we are `remembering an Eulerian tour by its vertices'. Although there may be other edges in $\Gamma$ which share the same source with $e$, $e^\prime$ is the only one which (roughly speaking) has the same source for purely formal reasons.

Define $i_e^{-1}(f)=f^\ast$, so that $i_e^{-1}$ is functorial $\text{Bypass}_S^\text{Eul}\rightarrow\Lambda^\text{op}$. We claim that $i_e^{-1}$ is inverse to $i_e$ when $S=\ast$.

In this case, $i_e^{-1}i_e(\mathbb{T}_m)$ is the cyclic set of elementary edges of $\mathbb{T}_m$, and if $f:\mathbb{T}_m\rightarrow\mathbb{T}_n$, then $i_e^{-1}i_e(f)$ sends an elementary edge $v_i\xrightarrow{e_i} v_{i+1}$ of $\mathbb{T}_m$ to the first elementary edge whose source is the same as $f(v_i)$.

In other words, there is a natural equivalence $i_e^{-1}i_e(\mathbb{T}_m)\to\mathbb{T}_m$ given by relabeling the vertices (labeled by elementary edges $e_i$) by vertices $v_i$. Similarly, $i_e i_e^{-1}(\mathbb{T}_m)\cong\mathbb{T}_m$, so $i_e,i_e^{-1}$ are inverse, completing the proof.
\end{proof}

\noindent The cyclic category also a few other useful properties, which we review now. We will say that:

\begin{definition}
A \emph{cyclic space} $X$ is a presheaf on $\Lambda$. Call $\text{cTop}=\mathcal{P}(\Lambda)$.
\end{definition}

\noindent First, notice that there is a functor $i:\Delta\rightarrow\Lambda$ from the simplex category (of finite, nonempty, totally ordered sets), which sends the totally ordered set $\{0<\cdots<n\}$ to the cyclically ordered set $\{0<\cdots<n<0\}$. Hence we may regard any cyclic space $X$ as having an underlying simplicial space $i^\ast X$.

We write $|X|$ for the geometric realization of the simplicial space. If $c$ denotes the functor $\Delta\rightarrow\ast$, then by definition $$|X|=c_\ast i^\ast X.$$ (For our pushforward and pullback notation, see Section \ref{SNot}.)

\begin{lemma}[\cite{DHK} Proposition 2.7]\label{LemDHK}
If $\mathbb{T}_n\in\Lambda\subseteq\text{cTop}$ denotes the representable cyclic space, then $|\mathbb{T}_n|\cong S^1$, and the functor $c_\ast i^\ast:\mathbb{T}_n\rightarrow\text{Top}$ factors through the subcategory $BS^1\subseteq\text{Top}$ of $S^1$-torsors.
\end{lemma}

\noindent Hence we have a functor $r:\Lambda\rightarrow BS^1$.

\begin{proposition}\label{PropSquare}
In the commutative square $$\xymatrix{
\Delta\ar[r]^i\ar[d]_c &\Lambda\ar[d]^r \\
\ast\ar[r]_i &BS^1,
}$$ there is a natural equivalence $i^\ast r_\ast\cong c_\ast i^\ast$ of functors $\text{cTop}\rightarrow\text{Top}$.
\end{proposition}

\begin{proof}
Lemma \ref{LemDHK} asserts that $i^\ast r_\ast\cong c_\ast i^\ast$ when restricted to $\Lambda\subseteq\mathcal{P}(\Lambda)$. Moreover, $r_\ast,c_\ast$ each have right adjoints ($r^\ast,c^\ast$), and each $i^\ast$ has right adjoint (given by right Kan extension). Therefore, $i^\ast r_\ast$ and $c_\ast i^\ast$ each preserve colimits by the adjoint functor theorem. Since all presheaves are colimits of representables, the proposition follows.
\end{proof}

\begin{corollary}\label{ColimLambda}
If $X\in\text{cTop}$, then $|X|$ has a canonical $S^1$-action, and $$\text{colim}_{\Lambda^\text{op}}(X)\cong|X|_{hS^1}.$$
\end{corollary}

\begin{proof}
By Proposition \ref{PropSquare}, $|X|\cong i^\ast r_\ast X$. As a presheaf over $BS^1$, $r_\ast X$ can be identified with an $S^1$-space, and $i^\ast r_\ast X$ with the underlying space (forgetting the $S^1$-action). Therefore, $|X|$ has a canonical $S^1$-action.

Moreover, if $j$ denotes the map $BS^1\rightarrow\ast$, then $j_\ast(-)\cong(-)_{hS^1}$ by Example \ref{ExColim}. Therefore, $$|X|_{hS^1}\cong j_\ast r_\ast X\cong (jr)_\ast X\cong\text{colim}(X).$$
\end{proof}

\begin{proposition}
The functor $r:\Lambda\rightarrow BS^1$ exhibits $BS^1$ as the classifying space of $\Lambda$.
\end{proposition}

\begin{proof}
We have $|\Lambda|\cong|\Lambda^\text{op}|\cong\text{colim}_{\Lambda^\text{op}}(\ast)$, the first equivalence by Lemma \ref{LemLambda} ($\Lambda\cong\Lambda^\text{op}$), and the second by \cite{HTT} 3.3.4.6. Once again, consider the functors $$\Lambda\xrightarrow{r}BS^1\xrightarrow{j}\ast.$$ Then $|\Lambda|\cong(jr)_\ast(\ast)\cong j_\ast r_\ast(\ast)$. By Proposition \ref{PropSquare}, $r_\ast(\ast)$ has underlying space the geometric realization of $\ast$, so $$|\Lambda|\cong j_\ast r_\ast(\ast)\cong j_\ast(\ast)\cong(\ast)_{hS^1}\cong BS^1.$$
\end{proof}

\section{The presheaf $\mathcal{O}_\text{thh}$}
\noindent In the next two sections, we will prove Theorem \ref{ThmCase1}, describing the presheaf $\mathcal{O}_\text{thh}$. We begin in this section by defining $\mathcal{O}_\text{thh}$.

Suppose that $\mathcal{V}$ is a closed symmetric monoidal, presentable $\infty$-category and $\mathcal{C}:\text{Bypass}_S\rightarrow\mathcal{V}$ is a $\mathcal{V}$-enriched category with set $S$ of objects. As in Section \ref{SNot}, we can extend $\mathcal{C}$ continuously to a functor $$\mathcal{C}_\ast:\mathcal{P}(\text{Bypass}_S)\rightarrow\mathcal{V}$$ which preserves colimits.

Let $(\mathcal{O}_\text{thh})_\bullet:\Lambda^\text{op}\rightarrow\mathcal{P}(\text{Bypass}_S)$ be the cyclic presheaf defined by $$(\mathcal{O}_\text{thh})_n=\coprod_{X_0,\ldots,X_n\in S}(X_0,X_1,\ldots,X_n,X_0),$$ and let $\mathcal{O}_\text{thh}\in\mathcal{P}(\text{Bypass}_S)$ be its geometric realization.

To be explicit, the cyclic structure is as follows: If $f:\mathbb{T}_m\rightarrow\mathbb{T}_n$ is a map in $\Lambda$, write $e_i$ for an elementary morphism of $\mathbb{T}_m$, and $f(e_i)=e_{i1}\cdots e_{ik}$ is a composite of elementary morphisms in $\mathbb{T}_n$. Then there is a bypass operation from $(X_0,\ldots,X_n,X_0)$ to some $(Y_0,\ldots,Y_m,Y_0)$ given by replacing each path $e_{i1}\cdots e_{ik}$ by a single edge, or by introducing a loop if $k=0$.

\begin{definition}[THH of an enriched category]\label{DefTHH}
If $\mathcal{C}:\text{Bypass}_S\rightarrow\mathcal{V}$ is a $\mathcal{V}$-enriched category with set $S$ of objects, then $\text{THH}(\mathcal{C})=\mathcal{C}_\ast(\mathcal{O}_\text{thh})$.
\end{definition}

\noindent In particular, because $\mathcal{C}_\ast$ preserves colimits, $\text{THH}(\mathcal{C})$ is the geometric realization $$\text{THH}(\mathcal{C})=\left\lvert[n]\mapsto\coprod_{X_0,\ldots,X_n}\mathcal{C}(X_0,X_1)\otimes\cdots\otimes\mathcal{C}(X_n,X_0)\right\rvert.$$

\begin{remark}
When $\mathcal{V}=\text{Sp}$, a $\mathcal{V}$-enriched category is a spectral category, and Definition \ref{DefTHH} is the usual cyclic bar construction computing THH (\cite{BM} Section 3).

When $\mathcal{V}=\text{Top}$, a $\mathcal{V}$-enriched category is an $\infty$-category, and Definition \ref{DefTHH} is the usual cyclic bar construction computing \emph{unstable THH} (\cite{Arbeit} pg. 857).
\end{remark}

\begin{remark}\label{RmkBar}
For any $X_0,\ldots,X_n\in S$, there is a canonical Eulerian tour on $(X_0,\ldots,X_n,X_0)$ given by the cycle $X_0\to\cdots\to X_n\to X_0$. Write $[X_0,\ldots,X_n,X_0]\in\text{Bypass}_S^\text{Eul}$ when we wish to remember the canonical Eulerian tour.

Then the cyclic presheaf $(\mathcal{O}_\text{thh})_\bullet$ lifts to a cyclic presheaf $$(\overline{\mathcal{O}}_\text{thh})_\bullet:\Lambda^\text{op}\rightarrow\mathcal{P}(\text{Bypass}_S^\text{Eul})$$ given by $(\overline{\mathcal{O}}_\text{thh})_n=\coprod_{X_0,\ldots,X_n\in S}[X_0,X_1,\ldots,X_n,X_0]$, such that $$\mathcal{O}_\text{thh}\cong k_\ast\overline{\mathcal{O}}_\text{thh}.$$ Here $k:\text{Bypass}_S^\text{Eul}\rightarrow\text{Bypass}_S$ is the forgetful functor.
\end{remark}

\section{THH of associative algebras}
\noindent In this section, we will identify $\mathcal{O}_\text{thh}$ when $S=\ast$. For the rest of this section, set $S=\ast$, and identify $\text{Bypass}_\ast^\text{Eul}\cong\Lambda^\text{op}$ via Lemma \ref{LemLambda}.

Let $\mathcal{Y}:\Lambda^\text{op}\times\Lambda\rightarrow\text{Top}$ denote the Yoneda map $\mathcal{Y}(X,Y)=\text{Map}(X,Y)$, which we may regard as a presheaf $\mathcal{Y}\in\mathcal{P}(\Lambda\times\Lambda^\text{op})$.

\begin{proposition}\label{PropOthh}
Consider the diagram $$\xymatrix{
\Delta\times\Lambda^\text{op}\ar[r]^c\ar[d]_i &\Lambda^\text{op}\ar[d]^i\ar@{=}[r] &\text{Bypass}_\ast^\text{Eul}\ar[r]^k &\text{Bypass}_\ast \\
\Lambda\times\Lambda^\text{op}\ar[r]_r &BS^1\times\Lambda^\text{op}, &&
}$$ where $k$ is the forgetful functor, and the square is $\Lambda^\text{op}$ times that of Proposition \ref{PropSquare}. Then $$\mathcal{O}_\text{thh}\cong k_\ast c_\ast i^\ast\mathcal{Y}\cong k_\ast i^\ast r_\ast\mathcal{Y}.$$
\end{proposition}

\begin{proof}
If $S=\ast$, then as in Remark \ref{RmkBar}, $(\overline{\mathcal{O}}_\text{thh})_n=[\ast,\cdots,\ast]$, which is $\mathbb{T}_n$ via the identification $\text{Bypass}_\ast^\text{Eul}\cong\Lambda^\text{op}$. Indeed, $(\overline{\mathcal{O}}_\text{thh})_\bullet$ is identified with the Yoneda embedding $$\Lambda^\text{op}\rightarrow\mathcal{P}(\Lambda^\text{op})\cong\mathcal{P}(\text{Bypass}_\ast^\text{Eul}).$$ Since $\overline{\mathcal{O}}_\text{thh}$ is the geometric realization, then $$\overline{\mathcal{O}}_\text{thh}\cong c_\ast i^\ast\mathcal{Y}.$$ By Lemma \ref{LemLambda}, we also have $$\overline{\mathcal{O}}_\text{thh}\cong i^\ast r_\ast\mathcal{Y}.$$ Since $\mathcal{O}_\text{thh}\cong k_\ast\overline{\mathcal{O}}_\text{thh}$ (Remark \ref{RmkBar}), the proposition follows.
\end{proof}

\begin{corollary}\label{CorBdl}
There is a canonical $S^1$-action on $\overline{\mathcal{O}}_\text{thh}$, and $(\overline{\mathcal{O}}_\text{thh})_{hS^1}\cong\ast$, the constant presheaf on $\text{Bypass}_\ast^\text{Eul}$.
\end{corollary}

\begin{proof}
This follows from Proposition \ref{PropOthh} and Corollary \ref{ColimLambda}.
\end{proof}

\noindent We will now turn to the first of our main results from the introduction. We can identify presheaves on $BS^1$ with $S^1$-equivariant spaces: $$\mathcal{P}(BS^1)\cong\mathcal{P}((BS^1)^\text{op})\cong\text{Top}^{S^1}.$$ Write $S^1\in\mathcal{P}(S^1)$ for the torsor (that is, $S^1$ acting freely on itself).

\begin{theorem}\label{ThmCase1}
Let $k:\Lambda^\text{op}\cong\text{Bypass}_\ast^\text{Eul}\rightarrow\text{Bypass}_\ast$ denote the forgetful functor and $r:\Lambda^\text{op}\rightarrow|\Lambda^\text{op}|\cong BS^1$. Then $$\mathcal{O}_\text{thh}\cong k_\ast r^\ast(S^1).$$
\end{theorem}

\begin{lemma}\label{LemCase1}
Suppose $\mathcal{O}\in\mathcal{P}(\Lambda^\text{op})$ has an $S^1$-action for which $\mathcal{O}_{hS^1}\cong\ast$ is the constant presheaf.

If $\text{colim}_\Lambda(\mathcal{O})\cong\ast$, then $\mathcal{O}\cong r^\ast(S^1)$.
\end{lemma}

\begin{proof}[Proof of lemma]
Since $\mathcal{O}$ has an $S^1$-action, it may be regarded as a functor $\mathcal{O}:\Lambda\rightarrow\text{Top}^{S^1}$ to $S^1$-spaces. Since $\mathcal{O}_{hS^1}\cong\ast$, it lands in the full subcategory spanned by the torsor $S^1$, which is equivalent to the $\infty$-groupoid $BS^1$. Hence we have $$\mathcal{O}:\Lambda\rightarrow BS^1\subseteq\text{Top}^{S^1}.$$ Since $BS^1$ is an $\infty$-groupoid, $\mathcal{O}$ factors through $|\Lambda|\cong BS^1$, as $$\Lambda\xrightarrow{r}BS^1\xrightarrow{f}BS^1\subseteq\text{Top}^{S^1}$$ for some $f:BS^1\rightarrow BS^1$. If $f$ is the degree $n$ map (multiplication by $n$ on $\pi_2\cong\mathbb{Z}$), then $\mathcal{O}\cong r^\ast(S^1_{(n)})$, where $S^1_{(n)}$ denotes the circle with $S^1$-action $$\theta\cdot z=\theta^n z.$$ Note that $S^1_{(n)}\cong S^1_{(-n)}$.

Let $c$ denote the functor $\Lambda\rightarrow\ast$ and $\mathcal{O}_n=r^\ast S^1_{(n)}$. We are to prove the following: If $c_\ast\mathcal{O}_n=\text{colim}_\Lambda(\mathcal{O}_n)\cong\ast$, then $n=\pm 1$.

First we show $n\neq 0$. Indeed, since $r^\ast$ has a right adjoint (right Kan extension), it preserves colimits, so $$\mathcal{O}_0\cong S^1 \otimes r^\ast(\ast),$$ which is the constant presheaf $S^1$. Thus $$c_\ast\mathcal{O}_0\cong S^1\times c_\ast(\ast)\cong S^1\times BS^1\neq\ast.$$ Thus $n\neq 0$. Moreover, since $\mathcal{O}_n\cong\mathcal{O}_{-n}$, assume $n>0$. Since $r^\ast$ preserves colimits, $\mathcal{O}_n\cong(\mathcal{O}_1)_{hC_n}$, with $C_n$ the cyclic subgroup of $S^1$ of order $n$. Hence $$\ast\cong c_\ast\mathcal{O}_n\cong c_\ast(\mathcal{O}_1)_{hC_n}.$$ As a left adjoint functor, $c_\ast$ also preserves colimits, so $c_\ast(\mathcal{O}_1)\cong C_n$.

But we also know that $(\mathcal{O}_1)_{hS^1}\cong r^\ast(\ast)\cong\ast$, the constant presheaf, so $$c_\ast(\mathcal{O}_1)_{hS^1}\cong BS^1.$$ Given that $c_\ast(\mathcal{O}_1)\cong C_n$, we know $(C_n)_{hS^1}\cong BS^1$, which implies $n=1$.
\end{proof}

\begin{proof}[Proof of Theorem \ref{ThmCase1}]
By Corollary \ref{CorBdl}, $(\overline{\mathcal{O}}_\text{thh})_{hS^1}\cong\ast$. By the lemma, we need only show that $\text{colim}(\overline{\mathcal{O}}_\text{thh})\cong\ast$. Consider the diagram $$\xymatrix{
\Delta\times\Lambda^\text{op}\ar[r]^c\ar[d]_i\ar[rd]^p &\Lambda^\text{op}\ar[r]^s &\ast \\
\Lambda\times\Lambda^\text{op}\ar[rd]_p &\Delta\ar[ru]_t\ar[d]^i & \\
&\Lambda &
}$$ Then $$\text{colim}(\overline{\mathcal{O}}_\text{thh})\cong s_\ast\overline{\mathcal{O}}_\text{thh}\cong s_\ast c_\ast i^\ast\mathcal{Y}\cong t_\ast p_\ast i^\ast\mathcal{Y}\cong t_\ast i^\ast p_\ast\mathcal{Y}.$$ But $p_\ast\mathcal{Y}(\mathbb{T}_n)\cong|(\text{Map}(\mathbb{T}_n,-))|_{hS^1}$ by Corollary \ref{ColimLambda}, which is $(S^1)_{hS^1}\cong\ast$ by Lemma \ref{LemDHK}. Hence $p_\ast\mathcal{Y}\cong\ast$, the constant presheaf, so $i^\ast p_\ast\mathcal{Y}\cong\ast$ is the constant simplicial space, which has contractible geometric realization.

The hypothesis of the lemma follows, so by the lemma $\overline{\mathcal{O}}_\text{thh}\cong r^\ast(S^1)$, and therefore $$\mathcal{O}_\text{thh}=k_\ast\overline{\mathcal{O}}_\text{thh}\cong k_\ast r^\ast(S^1).$$
\end{proof}

\section{THH of enriched categories}
\noindent Now we will identify $\mathcal{O}_\text{thh}\in\mathcal{P}(\text{Bypass}_S)$ for an arbitrary set $S$, generalizing Theorem \ref{ThmCase1} to Theorem \ref{ThmCase2}.

Let $r$ denote the composite $$\text{Bypass}_S^\text{Eul}\xrightarrow{F}\text{Bypass}_\ast^\text{Eul}\rightarrow|\text{Bypass}_\ast^\text{Eul}|\cong BS^1,$$ where $F$ forgets the labeling of the vertices. (We will later prove that $r$ exhibits $BS^1$ as the classifying space of $\text{Bypass}_S^\text{Eul}$; see Remark \ref{RmkBS1}.)

\begin{theorem}\label{ThmCase2}
Let the functor $r:\text{Bypass}_S^\text{Eul}\rightarrow BS^1$ be as above, and write $k:\text{Bypass}_S^\text{Eul}\rightarrow\text{Bypass}_S$ for the forgetful functor. Then $$\mathcal{O}_\text{thh}\cong k_\ast r^\ast(S^1).$$
\end{theorem}

\noindent Notice that this precisely recovers Theorem \ref{ThmCase1} when $S=\ast$. In general, we will derive this theorem from the $S=\ast$ case by means of a lemma asserting (roughly) that THH does not depend on the ambient set $S$.

Recall the presheaf $\overline{\mathcal{O}}_\text{thh}\in\mathcal{P}(\text{Bypass}_S^\text{Eul})$ of Remark \ref{RmkBar}. In this section, we will use a superscript (as in $\overline{\mathcal{O}}_\text{thh}^S$) to emphasize we are working over $\text{Bypass}_S$ for a particular $S$.

\begin{lemma}
If $F:\text{Bypass}_S^\text{Eul}\rightarrow\text{Bypass}_\ast^\text{Eul}$ is the functor which forgets the labeling of vertices, then $F^\ast\overline{\mathcal{O}}_\text{thh}^\ast\cong\overline{\mathcal{O}}_\text{thh}^S$.
\end{lemma}

\begin{proof}
Recall that $\overline{\mathcal{O}}_\text{thh}^S$ is the geometric realization of the cyclic presheaf $$(\overline{\mathcal{O}}_\text{thh}^S)_n=\coprod_{X_0,\ldots,X_n}[X_0,\ldots,X_n,X_0].$$ Given $\Gamma\in\text{Bypass}_S^\text{Eul}$, we may evaluate $$(\overline{\mathcal{O}}_\text{thh}^S)_n(\Gamma)\cong\coprod_{X_0,\ldots,X_n}\text{Map}(\Gamma,[X_0,\ldots,X_n,X_0]).$$ A map $\Gamma\rightarrow[X_0,\ldots,X_n,X_0]$ is an \emph{itinerary} of the Eulerian tour on $\Gamma$; that is, it is a way to regard the Eulerian tour as a tour with stops at $X_0,X_1,\ldots,X_n$ in that order.

Hence, $(\overline{\mathcal{O}}_\text{thh}^S)_n(\Gamma)$ is the set of all $n$-stop itineraries for the specified Eulerian tour.

On the other hand, $F^\ast(\overline{\mathcal{O}}_\text{thh}^\ast)_n(\Gamma)$ is the set of all $n$-stop itineraries for the Eulerian tour where we have forgotten the exact location of each stop.

In order to remember an itinerary, we need only remember the \emph{order} of our stops, not the exact location. Therefore, $(\overline{\mathcal{O}}_\text{thh}^S)_n\cong F^\ast(\overline{\mathcal{O}}_\text{thh}^\ast)_n$. Moreover, $F^\ast$ preserves all colimits (in particular geometric realizations), so $\overline{\mathcal{O}}_\text{thh}^S\cong F^\ast\overline{\mathcal{O}}_\text{thh}^\ast$, as desired.
\end{proof}

\begin{proof}[Proof of Theorem \ref{ThmCase2}]
Using Remark \ref{RmkBar}, $\mathcal{O}_\text{thh}^S\cong k_\ast\overline{\mathcal{O}}_\text{thh}^S$, which by the lemma and Theorem \ref{ThmCase1} is $$k_\ast F^\ast\overline{\mathcal{O}}_\text{thh}^\ast\cong k_\ast F^\ast r^\ast(S^1)\cong k_\ast r^\ast(S^1),$$ as desired. (There is a potential for confusion in the notation: We are using $r$ to refer first to $\text{Bypass}_\ast^\text{Eul}\rightarrow BS^1$, and second to $\text{Bypass}_S^\text{Eul}\rightarrow BS^1$, which is really the composite $rF$.)
\end{proof}

\section{Counting Eulerian tours two ways}
\noindent In section 3, we saw that $\text{Bypass}_\ast^\text{Eul}\cong\Lambda$. For sets $S\neq\ast$, there is still a close relationship between $\text{Bypass}_S^\text{Eul}$ and $\Lambda$, which we explore in this section.

Recall we are regarding the objects of $\Lambda$ as categories $\mathbb{T}_n$.

\begin{definition}
Given a category $\mathcal{C}$, its cyclic nerve is the cyclic set $$N_\text{cyc}\mathcal{C}=\text{Fun}(-,\mathcal{C}):\Lambda^\text{op}\rightarrow\text{Set}.$$
\end{definition}

\noindent Write $S_\text{triv}$ for the category with set $S$ of objects and exactly one morphism between each object (which is equivalent to the trivial category $\ast$). In this case, $(N_\text{cyc}S_\text{triv})_n\cong S^{n+1}$ is the set of labelings of $\{0,\ldots,n\}$ by $S$.

We now turn to the main result in this section. Notice that an object of $\text{Bypass}_S^\text{Eul}$ can be described in two ways:
\begin{enumerate}
\item as a nonempty graph $\Gamma\in\text{Bypass}_S$ along with an Eulerian tour; or
\item as a cyclically ordered set of edges $\mathbb{T}_n\in\Lambda$ along with a labeling of its vertices in $S$.
\end{enumerate}

\noindent Write $\text{Bypass}_S^{+}\subseteq\text{Bypass}_S$ for the full subcategory of nonempty graphs.

Categorically, (1) and (2) amount to the following two lemmas:

\begin{lemma}\label{L1}
The forgetful functor $\text{Bypass}_S^\text{Eul}\rightarrow\text{Bypass}_S^{+}$ is a right fibration, and the associated straightening $\text{Eul}:(\text{Bypass}_S^{+})^\text{op}\rightarrow\text{Set}$ sends a graph $\Gamma$ to its set $\text{Eul}(\Gamma)$ of Eulerian tours.
\end{lemma}

\begin{proof}
The (right fibrational) Grothendieck construction of Eul is the functor $\mathcal{G}\rightarrow\text{Bypass}_S^{+}$, where $\mathcal{G}$ is the category of nonempty graphs along with a chosen Eulerian tour. By construction, this right fibration is equivalent to $\text{Bypass}_S^\text{Eul}\rightarrow\text{Bypass}_S^{+}$.
\end{proof}

\begin{lemma}\label{L2}
The forgetful functor $$\text{Bypass}_S^\text{Eul}\rightarrow\text{Bypass}_\ast^\text{Eul}\cong\Lambda^\text{op}$$ is a left fibration, and the associated straightening is $N_\text{cyc}S_\text{triv}:\Lambda^\text{op}\rightarrow\text{Set}$, which sends $\mathbb{T}_n$ to the set of ways to label its vertices in $S$.
\end{lemma}

\begin{proof}
Recall the explicit description of the equivalence $i_e^{-1}:\text{Bypass}_\ast^\text{Eul}\rightarrow\Lambda^\text{op}$ from Lemma \ref{LemLambda}: $i_e^{-1}(\Gamma)$ is the set of edges of $\Gamma$, cyclically ordered by the Eulerian tour. If $f:\Gamma\rightarrow\Gamma^\prime$ is a bypass operation, then $i_e^{-1}(f):e\mapsto e^\prime$ if and only if $$f(e^\prime_{-})<e\leq f(e^\prime),$$ where $e^\prime_{-}$ is the edge immediately preceding $e^\prime$ in the Eulerian tour on $\Gamma$.

The composite $\text{Bypass}_\ast^\text{Eul}\cong\Lambda^\text{op}\rightarrow\text{Set}$ sends the graph $\Gamma_n$ with $n$ edges to the set of ways to label vertices in $S$.

The associated (left fibrational) Grothendieck construction is the functor $\mathcal{G}\rightarrow\text{Bypass}_\ast^\text{Eul}$, where $\mathcal{G}$ is defined as follows: An object is a graph $\Gamma_n\in\text{Bypass}_\ast^\text{Eul}$ with $n$ edges and for each edge $e$, a labeling $s(e)\in S$ which we call its \emph{source vertex}. A morphism is $f:\Gamma_m\rightarrow\Gamma_n$ such that:
\begin{itemize}
\item if the preimage of an edge $e\in\Gamma_n$ is the path $v_0\xrightarrow{e_0}\cdots\xrightarrow{e_{n-1}}v_n$, then $s(e)=s(e_0)$;
\item if the preimage of $e\in\Gamma_n$ is empty, and $f^{-1}(e)_{+}$ is the first edge in $\Gamma_m$ whose image under $f$ comes after $e$, then $s(e)=s(f^{-1}(e)_{+})$.
\end{itemize}
This is also a description of $\text{Bypass}_S^\text{Eul}$: the objects are all of the form $[X_0,X_1,\ldots,X_n,X_0]$, and a map is an order-presering function on the edge set which satisfies the properties above. Hence $\mathcal{G}\cong\text{Bypass}_S^\text{Eul}$, and this completes the proof.
\end{proof}

\noindent We end this section with two corollaries that will be important later.

\begin{corollary}\label{CorCS}
The classifying space of $\text{Bypass}_S^\text{Eul}$ is $|\text{Bypass}_S^\text{Eul}|\cong BS^1$.
\end{corollary}

\begin{proof}
By \cite{HTT} 3.3.4.6 and Lemma \ref{L2}, $|\text{Bypass}_S^\text{Eul}|\cong\text{colim}(N_\text{cyc}S_\text{triv})$. By Corollary \ref{ColimLambda}, this is equivalent to $|N_\text{cyc}S_\text{triv}|_{hS^1}$. In the particular case of $S_\text{triv}$, the cyclic nerve is identical to the ordinary nerve (as a simplicial set), and $|NS_\text{triv}|\cong\ast$ since $S_\text{triv}\cong\ast$. Therefore, $|\text{Bypass}_S^\text{Eul}|\cong(\ast)_{hS^1}\cong BS^1$.
\end{proof}

\begin{remark}\label{RmkBS1}
Notice that we have actually proven the stronger statement that the functor $\text{Bypass}_S^\text{Eul}\rightarrow\text{Bypass}_\ast^\text{Eul}\cong\Lambda$ is an equivalence on geometric realizations of nerves. Therefore, the functor $r$ of Theorem \ref{ThmCase2} exhibits $BS^1$ as the classifying space of $\text{Bypass}_S^\text{Eul}$.
\end{remark}

\begin{corollary}\label{CorComb}
The forgetful functor $k:\text{Bypass}_S^\text{Eul}\rightarrow\text{Bypass}_S^{+}$ induces an equivalence of $\infty$-categories $$k_\ast:\mathcal{P}(\text{Bypass}_S^\text{Eul})\rightarrow\mathcal{P}(\text{Bypass}_S^{+})_{/\text{Eul}}.$$
\end{corollary}

\noindent In particular, $k_\ast(\ast)\cong\text{Eul}$.

\begin{proof}
If $\mathcal{D}$ is a quasicategory, then $\mathcal{P}(\mathcal{D})\cong\text{RFib}(\mathcal{D})$ by \cite{BarwickShah} 1.4. Here $\text{RFib}(\mathcal{D})$ is the full subcategory of $(\text{Cat}_\infty)_{/\mathcal{D}}$ spanned by the right fibrations, and a presheaf $F$ is sent to its Grothendieck construction $\smallint F\rightarrow\mathcal{D}$.

Therefore, if $F$ is a presheaf of spaces on $\mathcal{D}$, the Grothendieck construction also induces an equivalence $$\mathcal{P}(\mathcal{D})_{/F}\cong\text{RFib}(\mathcal{D})_{/\smallint F}.$$ Now suppose that $\mathcal{D}$ is a 1-category and $F$ is discrete (a presheaf of sets). Then $\smallint F$ is a 1-category with an explicit model: an object is a pair $(X,x)$ with $X\in\mathcal{D}$ and $x\in F(X)$. A morphism $(X,x)\rightarrow(Y,y)$ is a morphism $X\xrightarrow{f}Y$ such that $F(f)(y)=x$.

Identifying categories with their nerves, the functor $\smallint F\rightarrow\mathcal{D}$ is a \emph{strict} right fibration; that is, given $n\geq 1$, $0<k\leq n$, and a diagram $$\xymatrix{
\Lambda^n_k\ar[r]\ar[d] &\smallint F\ar[d] \\
\Delta^n\ar[r] &\mathcal{D},
}$$ there is a \emph{unique} lift $\Delta^n\rightarrow\smallint F$. It follows that, if $G:\mathcal{A}\rightarrow\smallint F$ is a functor for which the composite $\mathcal{A}\rightarrow\mathcal{D}$ is a right fibration, then $G$ is a right fibration. (A simplex $\sigma$ in $\smallint F$ can be projected down to $\mathcal{D}$ and then lifted to some $\sigma^\prime$ in $\mathcal{A}$. Because the simplex in $\smallint F$ is suitably unique, $\sigma^\prime$ is a lift of $\sigma$.)

Therefore, $$\mathcal{P}(\mathcal{D})_{/F}\cong\text{RFib}(\mathcal{D})_{/\smallint F}\cong\text{RFib}(\smallint F)\cong\mathcal{P}(\smallint F).$$ When $\mathcal{D}=\text{Bypass}_S^{+}$ and $F=\text{Eul}$, we have the desired equivalence by Lemma \ref{L1}: $$\mathcal{P}(\text{Bypass}_S^{+})_{/\text{Eul}}\cong\mathcal{P}(\text{Bypass}_S^\text{Eul}).$$
\end{proof}

\section{THH as a circle bundle}
\noindent In this section, we will prove Theorem \ref{MainThm}, which identifies the presheaf $\mathcal{O}_\text{thh}\in\mathcal{P}(\text{Bypass}_S)$ explicitly, restricted away from the empty graph. It will turn out to be essentially an immediate corollary of Theorem \ref{ThmCase2} and Corollaries \ref{CorCS} and \ref{CorComb}.

First, it is not hard to calculate $\mathcal{O}_\text{thh}$ at the empty graph:

\begin{proposition}\label{OthhEasy}
There is an equivalence $\mathcal{O}_\text{thh}(\emptyset)\cong S$.
\end{proposition}

\begin{proof}
We know $$(\mathcal{O}_\text{thh})_n(\emptyset)\cong\coprod_{X_0,\ldots,X_n}\text{Map}(\emptyset,(X_0,\ldots,X_n,X_0)),$$ and $\text{Map}(\emptyset,(X_0,\ldots,X_n,X_0))$ is contractible if $X_0=\cdots=X_n$, and otherwise empty. Hence $(\mathcal{O}_\text{thh})_\bullet(\emptyset)$ is the constant cyclic set with value $S$, and therefore the geometric realization is $\mathcal{O}_\text{thh}(\emptyset)\cong S$.
\end{proof}

\noindent Henceforth, we will restrict to $\mathcal{O}^{+}_\text{thh}\in\mathcal{P}(\text{Bypass}_S^{+})$. By Theorem \ref{ThmCase2}, we know $\mathcal{O}^{+}_\text{thh}\cong k_\ast r^\ast(S^1)$, with functors as in $$\xymatrix{
\text{Bypass}_S^\text{Eul}\ar[r]^k\ar[d]_r &\text{Bypass}_S^{+} \\
BS^1. &
}$$ We will now prove our main result calculating $\mathcal{O}^{+}_\text{thh}$, divided into three propositions.

\begin{proposition}\label{MainThm1}
$\mathcal{O}_\text{thh}^{+}$ has a canonical $S^1$-action and $(\mathcal{O}_\text{thh}^{+})_{hS^1}\cong\text{Eul}$.
\end{proposition}

\begin{proof}
The circle action comes from the description $\mathcal{O}_\text{thh}^{+}\cong k_\ast r^\ast(S^1)$ of Theorem \ref{ThmCase2}. (As an $S^1$-bimodule, $S^1$ has an $S^1$-action even as an object of $\text{Top}^{S^1}$.) Since $k_\ast$ and $r^\ast$ each have right adjoints (given by $k^\ast$ and right Kan extension along $r$), they preserve colimits. Therefore, $$(\mathcal{O}_\text{thh}^{+})_{hS^1}\cong k_\ast r^\ast((S^1)_{hS^1})\cong k_\ast r^\ast(\ast)\cong k_\ast(\ast)\cong\text{Eul}.$$ The last step is by Corollary \ref{CorComb}.
\end{proof}

\begin{corollary}\label{CorOthh}
If $\Gamma\in\text{Bypass}_S$, then $$\mathcal{O}_\text{thh}(\Gamma)\cong\left\{\begin{array}{lr}
(S^1)^{\amalg\text{Eul}(\Gamma)}, &\text{if }\Gamma\neq\emptyset \\
S, &\text{if }\Gamma=\emptyset
\end{array}\right.,$$
\end{corollary}

\begin{proof}
The $\Gamma=\emptyset$ case is Proposition \ref{OthhEasy}, so assume $\Gamma\neq\emptyset$.

By Proposition \ref{MainThm1}, $$\mathcal{O}_\text{thh}(\Gamma)_{hS^1}\cong(\mathcal{O}_\text{thh})_{hS^1}(\Gamma)\cong\text{Eul}(\Gamma).$$ Therefore, $\mathcal{O}_\text{thh}(\Gamma)$ is an $S^1$-space with homotopy orbits equivalent to a \emph{set} $\text{Eul}(\Gamma)$. It must be that $\mathcal{O}_\text{thh}(\Gamma)$ is a disjoint union of circles indexed by $\text{Eul}(\Gamma)$.
\end{proof}

\noindent We call an $S^1$-equivariant presheaf $\mathcal{O}\in\text{Bypass}_S^{+}$ satisfying (as in Proposition \ref{MainThm1}) $\mathcal{O}_{hS^1}\cong\text{Eul}$ an \emph{$S^1$-bundle over Eul}.

In other words, an $S^1$-bundle over Eul is an $S^1$-equivariant object of $\mathcal{P}(\text{Bypass}_S^\text{Eul})\cong\mathcal{P}(\text{Bypass}_S^{+})_{/\text{Eul}}$ for which $\mathcal{O}_{hS^1}$ is the terminal object. Equivalently, it is a functor $$\mathcal{O}:(\text{Bypass}_S^\text{Eul})^\text{op}\rightarrow\text{Top}^{S^1}$$ which lands in the full subcategory $BS^1\subseteq\text{Top}^{S^1}$ spanned by the torsor $S^1$.

\begin{proposition}\label{MainThm2}
The moduli space of $S^1$-bundles over Eul is $BS^1\times\mathbb{Z}$.
\end{proposition}

\begin{proof}
By the discussion above, the moduli space of $S^1$-bundles over Eul is the full subcategory $$\text{Fun}((\text{Bypass}_S^\text{Eul})^\text{op},BS^1)\subseteq\text{Fun}((\text{Bypass}_S^\text{Eul})^\text{op},\text{Top}^{S^1}).$$ Since $BS^1$ is an $\infty$-groupoid, any functor $(\text{Bypass}_S^\text{Eul})^\text{op}\rightarrow BS^1$ factors through the classifying space, which by Corollary \ref{CorCS} is $|\text{Bypass}_S^\text{Eul}|\cong BS^1$.

Hence, the moduli space of $S^1$-bundles over Eul is $$\text{Map}(BS^1,BS^1)\cong BS^1\times\mathbb{Z}.$$
\end{proof}

\noindent As in the proof of Lemma \ref{LemCase1}, let $S^1_{(n)}\in\text{Top}^{S^1}$ denote the circle with its $S^1$-action $$\theta\cdot z=\theta^n z.$$ Writing $BS^1\xrightarrow{n} BS^1$ for the degree $n$ map, then $S^1_{(n)}=n^\ast(S^1)$.

The proof of Proposition \ref{MainThm2} asserts that every $S^1$-bundle $\mathcal{O}$ over Eul is of the form $$(\text{Bypass}_S^\text{Eul})^\text{op}\rightarrow|\text{Bypass}_S^\text{Eul}|\cong BS^1\xrightarrow{n}BS^1\subseteq\text{Top}^{S^1}$$ for some $n\in\mathbb{Z}$, or equivalently $$\mathcal{O}\cong r^\ast(S^1_{(n)})$$ for some $n\in\mathbb{Z}$, where $r:\text{Bypass}_S^\text{Eul}\rightarrow BS^1$ as usual. We call $n$ the \emph{degree} of $\mathcal{O}$.

\begin{proposition}\label{MainThm3}
As an $S^1$-bundle over Eul, $\mathcal{O}_\text{thh}^{+}$ has degree 1.
\end{proposition}

\begin{proof}
This is just the assertion that $\mathcal{O}_\text{thh}^{+}\cong k_\ast r^\ast(S^1)$, which is Theorem \ref{ThmCase2}, remembering that $k_\ast$ specializes to an equivalence $\mathcal{P}(\text{Bypass}_S^\text{Eul})\rightarrow\mathcal{P}(\text{Bypass}_S^{+})_{/\text{Eul}}$.
\end{proof}

\noindent Our main theorem is a summary of Propositions \ref{MainThm1}, \ref{MainThm2}, and \ref{MainThm3}:

\begin{theorem}\label{MainThm}
If $\mathcal{O}_\text{thh}^{+}$ is as above, then:
\begin{enumerate}
\item $\mathcal{O}_\text{thh}^{+}$ has a canonical $S^1$-action;
\item $(\mathcal{O}_\text{thh}^{+})_{hS^1}\cong\text{Eul}$;
\item The moduli space of presheaves satisfying (1)-(2) is equivalent to $\mathbb{Z}\times BS^1$; that is, such presheaves are determined up to equivalence by an integer invariant we call \emph{degree};
\item The degree of $\mathcal{O}^{+}_\text{thh}$ is $1$.
\end{enumerate}
\end{theorem}

\noindent Let $\text{Eul}(n)\in\mathcal{P}(\text{Bypass}_S^{+})$ denote the $S^1$-bundle of degree $n$ over Eul.

\begin{corollary}
If $\mathcal{V}$ is presentable and symmetric monoidal, and $\mathcal{C}:\text{Bypass}_S\rightarrow\mathcal{V}$ is a $\mathcal{V}$-enriched category, then $$\mathcal{C}_\ast\text{Eul}(\pm 1)\cong\text{THH}(\mathcal{C}),$$ $$\mathcal{C}_\ast\text{Eul}\cong\text{THH}(\mathcal{C})_{hS^1},$$ $$\mathcal{C}_\ast\text{Eul}(\pm n)\cong\text{THH}(\mathcal{C})_{C_n},\text{ if }n\geq 1,$$ $$\mathcal{C}_\ast\text{Eul}(0)\cong S^1\otimes\text{THH}(\mathcal{C})_{hS^1}.$$
\end{corollary}

\begin{proof}
The first statement summarizes the theorem. The later statements are because $\text{Eul}(n)\cong k_\ast r^\ast(S^1_{(n)})$, $\text{Eul}\cong k_\ast r^\ast(\ast)$, each of $r^\ast$, $k_\ast$, and $\mathcal{C}_\ast$ preserves colimits, and we have the identities in $\text{Top}^{S^1}$: $$\ast\cong(S^1_{(1)})_{hS^1}$$ $$S^1_{(n)}\cong(S^1_{(1)})_{hC_n}$$ $$S^1_{(0)}\cong S^1\otimes\ast.$$
\end{proof}

\end{document}